\documentclass[10pt,twocolumn]{IEEEtran} 

\usepackage{graphicx,amsmath,amsfonts,amsbsy,bm,amsthm,epstopdf} %amssymb,
\newtheorem{theorem}{Theorem}
\newtheorem{lemma}{Lemma}
\newtheorem{definition}{Definition}

\def\p{{ p }}
\def\e{ {\rm e}}
\def\bfnu{ {\boldsymbol \nu} }
\def\bfA{ {\boldsymbol A} }
\def\bfB{ {\boldsymbol B} }
\def\bfC{ {\boldsymbol C} }
\def\bfD{ {\boldsymbol D} }
\def\bfG{ {\boldsymbol G} }
\def\bfI{ {\boldsymbol I} }
\def\bfJ{ {\boldsymbol J} }

\def\bfs{ {\boldsymbol s} }
\def\bfM{ {\boldsymbol M} }
\def\bfO{ {\boldsymbol O} }
\def\bfP{ {\boldsymbol P} }
\def\bfQ{ {\boldsymbol Q} }
\def\bfZ{ {\boldsymbol Z} }
\def\bfU{ {\boldsymbol U} }
\def\bfV{ {\boldsymbol V} }
\def\bfu{ {\boldsymbol u} }
\def\bfv{ {\boldsymbol v} }

\def\bfUpsilon{ {\boldsymbol \Upsilon} }
\def\bfPsi{ {\boldsymbol \Psi} }
\def\bfOmega{ {\boldsymbol \Omega} }
\def\bfLambda{ {\boldsymbol \Lambda} }
\def\bfS{ {\boldsymbol S} }
\def\bfT{ {\boldsymbol T} }
\def\bfX{ {\boldsymbol X} }
\def\bfY{ {\boldsymbol Y} }
\def\deltt{\Delta_{\rm t}}
\def\EqualDist{\mathrel{\mathop=^{\rm d}}}
\def\EqualDef{\mathrel{\mathop=^{\rm def}}}
\def\bfSigma{  {\boldsymbol\Sigma}      }
\def\bfPhi{  {\boldsymbol\Phi}      }
\def\bfTheta{  {\boldsymbol\Theta}      }
\def\hSRB{ {\hat{\bfS}}_{\rm RB} }
\def\hSLW{ {\hat{\bfS}}_{\rm LW} }
\def\loading{ {\mu_0}}
\def\hatloading{ {\hat{\mu}}_0 }

\def\hrhoLW{ {{\hat{\rho}}_{\rm LW}} }
\def\hrhoRB{ {{\hat{\rho}}_{\rm RB}} }
\def\hbS{ {\hat{\bfS}} }

\newtheorem{remark}{Remark}

\def\Re{ {\rm Re}}

\def\eye{ {\rm i}}

\def\dif{ {\rm d} }

\def\bU{ {\boldsymbol U} }

\def\fN{f^{}_{\cal N}}

\def\EqualDist{\mathrel{\mathop=^{\rm d}}}

\def\var{\mathop{\rm var}\nolimits}

\def\cov{\mathop{\rm cov}\nolimits}

\def\tr{ \mbox{tr} }

\def\fN{f^{}_{\cal N}}

\begin{document}

%\markboth{IEEE Transactions on Signal Processing, [\today]}%{}

\title{Random Matrix Derived Shrinkage of Spectral Precision Matrices}

\author{A.~T.~Walden, {\it Senior~Member, IEEE},  and D.~Schneider-Luftman \thanks{
Copyright (c) 2014 IEEE. Personal use of this material is permitted. However, permission to use this material for
any other purposes must be obtained from the IEEE by sending a request to pubs-permissions@ieee.org.
%The work of Ed Cohen was supported by EPSRC (UK).
Andrew Walden and Deborah Schneider-Luftman
are both at the Department of Mathematics, Imperial College  London, 180 Queen's Gate,
London SW7 2BZ, UK.  
%Tel: (0)20 7594 8524 (Walden);
%Fax:   (0)20 7594 8517\newline
(e-mail: a.walden@imperial.ac.uk and deborah.schneider-luftman11@imperial.ac.uk)
%\newline
%EDICS: SSP-SSAN; SSP-NSSP 
}}

\maketitle

\begin{abstract} 
Much research has been carried out on shrinkage methods for real-valued covariance matrices. In spectral analysis of $p$-vector-valued time series there is often a need for good shrinkage methods too, most notably when the  complex-valued spectral matrix is singular. The equivalent of the Ledoit-Wolf (LW) covariance matrix estimator for spectral matrices can be improved on using a Rao-Blackwell estimator, and using random matrix theory we derive its form. 
Such estimators can be used to better estimate inverse spectral (precision) matrices too, and a  random matrix method  has previously been proposed and implemented via extensive simulations. We describe the method, but carry out computations entirely analytically,   and suggest a way of selecting an important parameter using a predictive risk approach. We show that both the Rao-Blackwell estimator  and the random matrix estimator of the precision matrix can substantially outperform the inverse of the LW estimator in a time series setting. 
Our new methodology is applied to EEG-derived time series data  where it is seen to work well and deliver substantial improvements for precision matrix estimation.

\end{abstract}

\begin{IEEEkeywords} Rao-Blackwell estimators, random matrix theory, shrinkage, spectral matrix.
\end{IEEEkeywords}

%\maketitle
%
%
%
%
%
%\begin{abstract}
%
%%\boldmath
%
%The abstract goes here.
%
%\end{abstract}
%
%\begin{IEEEkeywords}
%
%IEEEtran, journal, \LaTeX, paper, template.
%
%\end{IEEEkeywords}

%\IEEEpeerreviewmaketitle

%%%%%%%%%%%%%%%%%%%%%%%%%%%%%%%%%%%%%%%%%%%%%%%%%%%%%%%%%%%%%%%%%%%%%%%%%%%%%%%%%%%%%%%%%%%%%%%%%%%%%%%%%

\section{Introduction}
\label{sec:introduction}

A stationary
$p$-vector-valued time series has, at each frequency $f,$ a $p\times p$ complex-valued  spectral matrix  $\bfS(f),$ for which an  
estimator $\hbS(f),$ can be derived.  If such an estimator is computed by a multitaper scheme involving $K$ tapers (e.g., \cite{PercivalWalden93}) then the spectral matrices --- complex-valued analogues of covariance matrices --- will be singular if $p>K$ 
(and ill-conditioned if $K$ is only a little larger than $p$). 
Unfortunately $K$ cannot be simply increased because of its connection to the implied smoothing bandwidth:  if $K$ is made larger, the required resolution may be lost. (Other estimators such as periodograms smoothed over frequencies have analogous properties.) In this paper we look at the estimation of $\bfS(f)$ and more particularly the spectral `precision' matrix defined as 
$\bfC(f)=\bfS^{-1}(f)$ when $\hbS(f)$ is singular. The precision matrix is used in the computation of partial coherencies in time series graphical modelling (see e.g. \cite{Medkouretal10}
and references therein for a neuroscience application). We don't assume a very large $p$ since the moderate $p$ scenario is often encountered in practice and practically is just as  important. We shall first give a review of relevant covariance matrix estimation literature, before turning to the contributions of this paper.

The estimation of a  covariance matrix $\bfSigma$ from $N$ samples  of $p$ real-valued zero mean random variables has been extensively researched for the case $N>p.$ 
%with $N$ of the same order as $p$ with $p$ not particularly large. 
Although the resulting non-singular sample covariance estimator ${\hat{\bfSigma}}$ of $\bfSigma$ is unbiased
its eigenvalues tend to be more
spread out than the true eigenvalues. To ameliorate this problem
\cite{JamesStein61} looked at minimax estimation over a certain group, but the estimators depend on the coordinate system.
This problem was removed by
\cite{DeySrinivasan85} who considered orthogonally equivariant minimax estimators:
an  estimator ${\cal F}({\hat\bfSigma})$ of $\bfSigma$ is said to be orthogonally equivariant
 if for any orthogonal matrix $\bfO,$ we have
$
{\cal F}(\bfO{\hat\bfSigma}\bfO^T)=\bfO{\cal F}({\hat\bfSigma})\bfO^T,  
$
where $^T$ denotes transposition.
In fact such estimators shrink the sample eigenvalues, and so are of the widely researched shrinkage class, see e.g., \cite{EfronMorris76,Haff80,Stein75}. 

For shrinkage estimators which are a combination of the standard covariance matrix and a target matrix proportional to the identity, Ledoit and Wolf (LW) \cite{LedoitWolf03,LedoitWolf04} derived the ideal shrinkage parameter, or `oracle' value, that minimizes a risk measure between ${\hat{\bfSigma}}$ and $\bfSigma.$ Such LW estimators are (i) suitable for the case $N<p$ when ${\hat{\bfSigma}}$ is singular, (ii) do not assume Gaussianity, and (iii) may be used in large $p$ settings. Modifications to the target matrix were discussed in \cite{SchafferStrimmer05} and \cite{ChenWangMcKeown12}, the latter shrinking the sample covariance matrix towards its tapered version for high-dimensional matrices; modified estimators for this case  were also suggested in \cite{FisherSun11}.

Under the Gaussianity assumption, \cite{Chen_etal10} showed that the LW estimator can be significantly improved upon. They developed the so-called Rao-Blackwell (RB) estimator which is guaranteed at least as good as the LW estimator under any convex loss criterion.

There has also been much interest in accurate estimation of the precision matrix $\bfSigma^{-1}.$  A weighted combination of ${\hat{\bfSigma}}^{-1}$ and the identity was considered by \cite{EfronMorris76}, and improved on by \cite{Haff79}. By looking over the class of orthogonally equivariant estimators for real covariance matrices, Ledoit and Wolf \cite{LedoitWolf12} produced nonlinear shrinkage estimators for $\bfSigma$ and $\bfSigma^{-1}.$
All these studies assumed that $N>p.$ 
Also the  calculations involved in \cite{LedoitWolf12} are hugely costly. 
The singular case has been attracting much attention recently in the context of estimating sparse precision matrices $\bfSigma^{-1}$ in high-dimensional situations ($p>>N$), see e.g., \cite{Cai_etal11,Chen_etal12,LamFan09,MeinhausenBuhlman06,Rothman_etal08}.

%{Zhang_etal13}{KubokawaSrivastava08} 

Following some background material on spectral matrix estimation in Section~\ref{sec:backg},
the  contributions of this paper are as follows.
\begin{enumerate}
\item
In Section~\ref{sec:conventional} we study LW oracle estimation for $\bfS(f),$ and give the form of the practical estimator $\hSLW(f).$ The related Rao-Blackwell estimator for the spectral matrix, $\hSRB(f)$, is found 
in Section~\ref{sec:RaoB}.  These oracle and Rao-Blackwell estimators are surprisingly different in form to the real-valued cases. The Rao-Blackwell estimator is derived making substantial use of random matrix theory and is very simple in form and thus highly usable in practice. The Gaussian assumption is used to derive simple forms for the oracle shrinkage parameter and for the Rao-Blackwell estimator. While in standard real-valued covariance matrix estimation Gaussianity is a problematic assumption and robustness issues arise, in our context this is not dubious because of the Central Limit Theorem effect of the vector Fourier transform used in the time series setting.
\item
Section~\ref{sec:invRaoB} points out that
the inverse of the Rao-Blackwell estimator is in the form of a ``Rao-Blackwellized'' estimator for $\bfC(f).$
We show that this estimator can substantially outperform the inverse of the LW estimator  in a time series setting. 
\item
In Section~\ref{sec:Marzetta} we examine direct estimation of $\bfC(f)$ from singular estimators
${\hat\bfS}(f)$  using random matrix methods as developed in \cite{Marzetta_etal11}, and 
formulate a completely analytic (rather than simulation-based) approach to obtain the estimators.
A predictive risk approach is given to select a controlling parameter. 
We show that this estimator can substantially outperform the inverse of the LW estimator  in a time series setting. 
\item
Our new methodology is applied to electroencephalogram (EEG) derived time series data in Section~\ref{eq:EEGapp}, where it is seen to work well and deliver substantial improvements over the 
inverse LW estimators of $\bfC(f).$

\end{enumerate}

\section{Spectral Matrix Estimation}  \label{sec:backg}
Here we consider a  real  $p$-vector-valued
discrete time stochastic process\/ $\{ {\bfX}_t \}$ 
whose $t$th element is the column vector 
${\bfX}_t=[X_{1,t},\ldots,X_{p,t}]^T,$ 
and  each component process has zero mean. The sample interval is denoted by $\deltt.$
We assume the $\p$ processes are jointly stationary, i.e., for all $l,m=1,\ldots,p,$
$s_{lm,\tau}=\cov\,\{ X_{l,t+\tau}, X_{m,t} \}$
is a function of $\tau$ only.

The matrix autocovariance sequence $\{ \bfs_\tau\}$ is defined by
${\bfs}_{\tau}=\cov\{  {\bfX}_{t+\tau}, {\bfX}_t^T\}=
E\{ {\bfX}_{t+\tau} {\bfX}_t^T \},
$
and each component is assumed absolutely summable.
The spectral matrix,  is then
$
\bfS(f)=\deltt \sum_{\tau=-\infty}^\infty {\bfs}_{\tau}
\,\e^{-{\eye}2\pi f \tau \,\deltt}. 
$ 

We make use of a set of $K$ orthonormal tapers $\{h_{k,t}\}, k=0,\ldots,K-1$
and for $t=0,\ldots,N-1,$
form the product
$h_{k,t} {\bfX}_t$ of the $t$th component of the $k$th taper with the $t$th component of the $p$-vector-valued
process, and for $k=0,\ldots,K-1$ compute the  vector Fourier transform
$$
{\bfJ}_{k}(f)
\EqualDef
\deltt^{1/2}\sum_{t=0}^{N-1} h_{k,t} {\bfX}_t \,\e^{-\eye 2\pi ft\,\deltt}.
%\eqno\EqFigDef\eqnSEgg
$$

Let $\bfJ(f)$ be the $p\times K$ matrix defined by
\begin{equation}\label{eq:defJ}
\bfJ(f)=[\bfJ_0(f),\ldots,\bfJ_{K-1}(f)].
\end{equation}
Then the multitaper estimator of the $p\times p$ spectral matrix ${\bfS}(f)$ 
is
\begin{eqnarray}
{\hat {\bfS}} (f) &=&
{1 \over K} \sum_{k=0}^{K-1}  {\hbS}_k(f)=
{1 \over K} \sum_{k=0}^{K-1}  {\bfJ}_k(f) {
\bfJ}^H_k(f)= \nonumber \\
&=&
 {1 \over K} \bfJ(f)\bfJ^H(f), 
\label{eq:eqnSHRdd}
\end{eqnarray}
where ${\hbS}_k(f){\displaystyle{\EqualDef}}{\bfJ}_k(f) {\bfJ}^H_k(f).$

\begin{remark}
This conveniently mimicks the classical covariance matrix estimator: 
if ${\bfY}_0,\ldots,{\bfY}_{K-1}$ are $K$ independent $p$-dimensional Gaussian real-valued random vectors  with zero means and covariance matrix $\bfSigma,$
 then the maximum likelihood estimator for $\bfSigma$ is
$
{\hat{\bfSigma}}={1\over K} \sum_{k=0}^{K-1} \bfY_k\bfY_k^T.
%\eqno\EqFigDef\eqnSHRbb
$
\end{remark}
\medskip

Letting $B$ denote the bandwidth of the spectral window corresponding to the tapering, then
${\bfJ}_{k}(f), k=0,\ldots,K-1,$  
may be taken to be  independently and 
identically distributed 
as $p$-vector-valued complex Gaussian with mean zero and covariance matrix $\bfS(f):$
\begin{equation}
{\bfJ}_{k}(f)
 \EqualDist
{\cal N}_p^C\{{\bf 0},{\bfS}(f)\},
\label{eq:eqnSHRee}
\end{equation}
for $B/2< |f|< \fN-B/2$ for finite $N$ and Gaussian processes, or $0<|f|<\fN$ asymptotically \cite{ChandnaWalden11}. 
Then the estimator of (\ref{eq:eqnSHRdd}) is the maximum-likelihood estimator for $\bfS(f),$ \cite{Goodman63}. 
Further,
\begin{equation}
E\{ {\hat {\bfS}} (f)\}\!=\!
{1 \over K} \sum_{k=0}^{K-1} E\{ {\bfJ}_k(f) {
\bfJ}^H_k(f)\}\!=\!{1 \over K} \sum_{k=0}^{K-1}  {\bfS}(f)=\bfS(f), 
\label{eq:eqnSHRll}
\end{equation}
and $E\{ \tr\{ {\hat{\bfS}}\} \}=E\{ \sum_{j=1}^p {\hat S}_{jj}\}= \sum_{j=1}^p S_{jj}=\tr\{\bfS\},$  results we shall make use of later.
These hold whether $K\geq p,$ which corresponds to ${\hat {\bfS}} (f)$ being non-singular,  or $K<p,$ when the estimated matrix is singular (both with probability one).

\section{Conventional Shrinkage Methodology}\label{sec:conventional}
The conventional
approach to `covariance matrix' regularization which has been extensively studied involves the forming of a convex combination of the sample covariance matrix and some well-conditioned `target' matrix. For an estimated $p\times p$ Hermitian spectral matrix ${\hat{\bfS}}(f)$ this would take the form
\begin{equation}
{\bfS}^{\star}(f)=(1-\rho(f)) {\hat{\bfS}}(f)+\rho(f){\hat{\bfT}}(f),
\label{eq:eqnSHRff}
\end{equation}
where $\rho(f)\in(0,1)$ is known as the shrinkage parameter and ${\hat{\bfT}}(f)$ is the target matrix. Provided 
${\hat {\bfS}}(f)$ and ${\hat{\bfT}}(f)$ are both positive definite, then this convex combination   will itself be positive definite. For notational brevity we shall drop the explicit frequency dependence in most of what follows.

Apart from being positive definite, suppose that no {\it a priori}\/ form is imposed on ${\hat{\bfT}}$ and our goal is to find an optimal estimator for $\bfS$ of the form of (\ref{eq:eqnSHRff}) by determining $\rho=\rho_0$ such that
$$
\rho_0=\arg\min E\{ || {\bfS}^{\star} -\bfS ||^2_{\rm F}\},
$$ 
where, for $\bfA\in{\mathbb C}^{p\times p}$, $|| \bfA ||_{\rm F}$ denotes the Frobenius norm
$
||\bfA||_{\rm F}=[\tr\{\bfA \bfA^H\}]^{1/2}
,$ $\tr\{\cdot\}$ denotes trace,
and $^H$ denotes complex-conjugate (Hermitian) transpose.

\subsection{Oracle Estimator}
Firstly we define
\begin{eqnarray*}
\alpha^2 &=& E\{ || \bfS -{\hat{\bfT}}\, ||^2_{\rm F}\}=E\{\tr\{[\bfS -{\hat{\bfT}}\,][{\bfS -{\hat{\bfT}}}\,]^H\} \}\\
\beta^2 &=& E\{ || {\hat{\bfS}} -\bfS ||^2_{\rm F}\}=E\{\tr\{[ {\hat{\bfS}} -\bfS][ {\hat{\bfS}} -\bfS]^H\}\} \\
\delta^2 &=& E\{ || {\hat{\bfS}} -{\hat{\bfT}}\, ||^2_{\rm F}\}=E\{\tr\{[ {\hat{\bfS}} -{\hat{\bfT}}\, ][ {\hat{\bfS}} -{\hat{\bfT}}\, ]^H\}\} \\
\gamma^2 &=& E\{\tr\{[ {\hat{\bfS}} -\bfS ][ \bfS-{\hat{\bfT}}\,  ]^H\}\}.
\end{eqnarray*}

Then with $\Re\{\cdot\}$ denoting ``real part of,''
\begin{eqnarray*}
\delta^2 &=& E\{ || {\hat{\bfS}} -{\hat{\bfT}}\, ||^2_{\rm F}\}=
E\{ ||\,\, [{\hat{\bfS}} -\bfS] + [\bfS -{\hat{\bfT}}\,]\,\,||^2_{\rm F}\}\\
&=& E\{ ||\,\,\bfS -{\hat{\bfT}}\,\,||^2_{\rm F}\}+E\{ ||\,\, {\hat{\bfS}} -\bfS \,\,||^2_{\rm F}\}\\
&+&
2\Re\{ E\{\tr\{[{\hat{\bfS}} -\bfS]
[\bfS -{\hat{\bfT}}\,]^H\}\}\}\\
&=&\alpha^2+\beta^2+2\gamma^2,
\end{eqnarray*}
since $[{\hat{\bfS}} -\bfS]$ and $[\bfS -{\hat{\bfT}}\,]^H$ are both Hermitian, (each of $\hbS, \bfS$ and ${\hat{\bfT}}$ is Hermitian), and therefore the trace of the product is guaranteed real-valued,  so $\Re\{\cdot\}$ is not needed.

The objective function can be written
\begin{eqnarray*}E\{ || {\bfS}^{\star}-\bfS ||^2_{\rm F}\}&=&
E\{ || (1-\rho) {\hat{\bfS}}+\rho{\hat{\bfT}}-\bfS||^2_{\rm F}\}\\
&=& E\{ || \,\,\rho[ {\hat{\bfT}}-\bfS]
+(1-\rho)[ {\hat{\bfS}}-\bfS]\,\,||^2_{\rm F}\}\\
&=& \rho^2 \alpha^2+ (1-\rho)^2\beta^2-2\rho(1-\rho)\gamma^2.
\end{eqnarray*}

Differentiating with respect to $\rho$ and setting to zero:
$$
{\partial\over \partial\rho} E\{ || {\bfS}^{\star} -\bfS ||^2_{\rm F}\}=2\rho\alpha^2-2(1-\rho)\beta^2-2(1-2\rho)\gamma^2=0
$$
so that the solution is \cite{FiecasOmbao11,FisherSun11}
\begin{equation}
\rho_0={{\beta^2+\gamma^2}\over{\delta^2}}={{\beta^2-\alpha^2+\delta^2  }\over{2\delta^2  }}.
\label{eq:eqnSHRjj}
\end{equation}
The second derivative is positive so that the objective function is minimized with this $\rho_0$ value. 

The term ${\beta^2+\gamma^2}$ can be rewritten as
\begin{eqnarray*}
\!\!\!&&\!\!\!\!\!\!\!\!\!
E\{\tr\{[ {\hat{\bfS}} -\bfS][ {\hat{\bfS}} -\bfS]^H\}\} +E\{\tr\{[ {\hat{\bfS}} -\bfS ][ \bfS-{\hat{\bfT}}\,  ]^H\}\}\\
&=&E\{\tr\{[ {\hat{\bfS}} -\bfS][ {\hat{\bfS}} -{\hat{\bfT}}]\}\},
\end{eqnarray*}
where we have used the Hermitian properties of ${\hat {\bfS}}$ and ${\hat{\bfT}}.$ So $\rho_0$ in (\ref{eq:eqnSHRjj}) becomes
\begin{equation}
\rho_0=
{
{
E\big\{\tr\{[ {\hat{\bfS}} -\bfS][ {\hat{\bfS}} -{\hat{\bfT}}]\}\big\}
}
\over
{
E\big\{\tr\{[ {\hat{\bfS}} -{\hat{\bfT}}\, ]^2\}\big\}
}
}.\label{eq:eqnSHRaf}
\end{equation}
which is of the same form as found in \cite[eqn.~(6)]{Chen_etal10} for the real-valued case.
This form for $\rho_0$ is distribution invariant.
In order to rewrite $\rho_0$ in (\ref{eq:eqnSHRaf}) in a useful form involving just $\bfS$ and parameters $K$ and $p,$ Gaussianity will be assumed, which is justified as discussed earlier.

\subsection{Stochastic Target}\label{subsec:stochastic}

Suppose we define $\loading=\tr\{\bfS\}/p$ and $
{\hatloading}=\tr\{{\hat{\bfS}}\}/p
%\eqno\EqFigDef\eqnSHRoo
$
 and take
$
{\hat{\bfT}}=(\tr\{ \hbS \}/p)\bfI_p= {\hatloading} \bfI_p.
$
In this case both ${\hat{\bfT}}$ and $\hbS$ will be subject to estimation error and will in general be correlated. (This was the case developed in \cite{LedoitWolf03} for  real-valued covariance matrices.)
\begin{theorem}\label{theorem:one}
Let ${\hat{\bfT}}=(\tr\{ \hbS \}/p)\bfI_p.$
Under the assumption (\ref{eq:eqnSHRee}), $\rho_0$ in (\ref{eq:eqnSHRaf})
can be written
\begin{equation}
\rho_0={
{
\tr^2\{\bfS\}-{1\over p} \tr\{ \bfS^2 \}
}
\over
{
[1-{K\over p}]\tr^2\{\bfS\}+
[K-{1\over p}]\tr\{\bfS^2\}
}
}.\label{eq:eqnSHRae}
\end{equation}
\end{theorem}
\begin{proof}
From (\ref{eq:eqnSHRaf})
$$
\rho_0=
{
{
E\big\{\tr\{[ {\hat{\bfS}} -\bfS][ {\hat{\bfS}} -(\tr\{ \hbS \}/p)\bfI_p]\}\big\}
}
\over
{
E\big\{\tr\{[ {\hat{\bfS}} -(\tr\{ \hbS \}/p)\bfI_p\, ]^2\}\big\}
}
}.
$$ 
The numerator and denominator are then
\begin{eqnarray*}
&&E\big\{ \tr\{ \hbS^2\}-{1\over p}\tr^2\{ \hbS\} -\tr\{\bfS\hbS\}+{1\over p} \tr\{\bfS\}\tr\{\hbS\} \big\}\\
\mbox{and}&&E\big\{\tr\{\hbS^2\}-{1\over p}\tr^2\{ \hbS\} \big\},
\end{eqnarray*}
respectively. Under the assumption (\ref{eq:eqnSHRee}),  $
K\hbS$ has the complex Wishart distribution with mean $K\bfS.$ Then we know (e.g., \cite{MaiwaldKraus00})
\begin{eqnarray*}
E\big\{ \tr\{ \hbS^2\} \big\}&=& \tr\{ \bfS^2\}+ {1\over K} \tr^2\{\bfS\}\\
E\big\{ \tr^2\{ \hbS\} \big\}&=& \tr^2\{ \bfS\}+ {1\over K} \tr\{\bfS^2\}.
\end{eqnarray*}
So the numerator and denominator become
\begin{eqnarray*}
&&{1\over K} [ \tr^2\{\bfS\}-{1\over p}\tr\{\bfS^2\}]\\
\mbox{and}&&\left[1-{1\over {pK}}\right]\tr\{\bfS^2\}+\left[{1\over K}-{1\over p}\right]\tr^2\{\bfS\},
\end{eqnarray*}
respectively, and their ratio gives the required result.
\end{proof}
The form (\ref{eq:eqnSHRae}) is known as an `oracle' estimator 
since it involves the unknown quantities $\tr\{\bfS\}$ and $\tr\{\bfS^2\}$
and so its value is not known in practical situations.
\begin{remark}
The form of the estimator (\ref{eq:eqnSHRae}) for complex-valued covariance matrix estimators is  surprisingly different to that for real-valued covariance matrix estimators: compare (\ref{eq:eqnSHRae}) with 
\cite[eqn.~(7)]{Chen_etal10}.

\end{remark}

\subsection{Deterministic Target}\label{subsec:deterministic}
If ${\hat{\bfT}}$ is  constant, ${\hat{\bfT}}=\bfT$ say, then the term $\gamma^2=
\tr\{E\{[ {\hat{\bfS}} -\bfS ]\}[ \bfS-{{\bfT}}\,  ]\}=0,$ and $\rho_0$ in
(\ref{eq:eqnSHRjj}) becomes
$\rho_0=\beta^2/\delta^2.$ 
We now consider the target matrix  $\bfT=(\tr\{ \bfS \}/p)\bfI_p=\loading\bfI_p.$  
\begin{theorem}
Let
$\bfT=(\tr\{ \bfS \}/p)\bfI_p.$
Under the assumption (\ref{eq:eqnSHRee}), 
$
\rho_0={{\beta^2}/{\delta^2}}
$
can be written
\begin{equation}
\rho_0={
{
\tr^2\{\bfS\}
}
\over
{
[1-{K\over p}]\tr^2\{\bfS\}+
K\tr\{\bfS^2\}
}
}.\label{eq:eqnSHRext}
\end{equation}
\end{theorem}
\begin{proof}
This proceeds along the same lines as for Theorem~\ref{theorem:one}.
\end{proof}

This case was extensively studied in \cite{LedoitWolf04}
who made many interesting observations. 
When $\bfT=\loading\bfI_p,$ then using (\ref{eq:eqnSHRff})
the eigenvalues of $\hbS$ are shrunk according to 
$ 
{\hat\lambda}_i \rightarrow (1-\rho_0) {\hat\lambda}_i+\rho_0\loading,
$ 
thus reducing the condition number. $\mu_0$ is the ``grand mean'' of both true and sample eigenvalues \cite{LedoitWolf04} and thus the  sample eigenvalues will be shrunk towards their grand mean.
In practice we will  know neither $\loading$ nor  $\rho_0=\beta^2/\delta^2$ since they both involve the unknown $\bfS.$ These quantities can be estimated via ``plug-in'' values. Following the derivation of consistent estimators in \cite{BohmvSachs09} 
we first take ${\hatloading}$ for $\mu$  
and next note that $\delta^2$ could be estimated by omitting the expected value:
\begin{eqnarray*}
{\hat{\delta^2}} &=&
 || {\hat{\bfS}} -{\hatloading}\bfI_p ||^2_{\rm F}=\tr\{[ {\hat{\bfS}} -{\hatloading}\bfI_p ][ {\hat{\bfS}} -{\hatloading}\bfI_p ]^H\}\\
&=&\tr\{ {\hat{\bfS}}^2 \}-{\tr^2\{ {\hat{\bfS}}\}\over p}
%\EqFigDef\eqnSHRgg
=\sum_{i=1}^p\sum_{j=1}^p |{\hat S}_{ij}-{\hatloading}\delta_{i,j}|^2,
%\EqFigDef\eqnSHRhh
\end{eqnarray*}
where $\delta_{i,j}$ is the usual Kronecker delta, equal to unity when $i=j,$ and zero otherwise. The estimation of $\beta^2=E\{ || {\hat{\bfS}} -\bfS ||_{\rm F}^2\}$ is less simple. Using (\ref{eq:eqnSHRll}), $\beta^2$ can be written
\begin{equation}
\beta^2=\sum_{i=1}^p\sum_{j=1}^p E\{ |{\hat{S}}_{ij}-E\{{\hat{S}}_{ij}\}|^2\}= \sum_{i=1}^p\sum_{j=1}^p \var\{ {\hat{S}}_{ij} \},
\label{eq:eqnSHRii}
\end{equation}
so it can be estimated using a form of sample variance: ${\hat{\beta}}^2=
 \sum_{i=1}^p\sum_{j=1}^p {\widehat{\var}}\{ {\hat{S}}_{ij} \}.$ Given  (\ref{eq:eqnSHRee}), for the multitaper spectral matrix estimator  we know 
$
\var\{ {\hat{S}}_{ij}\}=\var\left\{ (1/K)\sum_{k=0}^{K-1} {\hat{S}}_{k,ij}\right\}=(1/K) \var\{ {\hat{S}}_{k,ij}\},
$
where ${\hat{S}}_{k,ij}=({\hat{\bfS}}_k)_{ij}.$ An estimator for 
$\var\{ {\hat{S}}_{k,ij}\}$ is 
$
{\widehat{\var}}\{ {\hat{S}}_{k,ij}\}= (1/K) \sum_{k=0}^{K-1}| {\hat{S}}_{k,ij}-{\hat{S}}_{ij}|^2,
$
so we get
$
{\widehat{\var}}\{ {\hat{S}}_{ij}\}= (1/K^2) \sum_{k=0}^{K-1}| {\hat{S}}_{k,ij}-{\hat{S}}_{ij}|^2,
$
which gives  an estimator of $\beta^2$ in (\ref{eq:eqnSHRii}) of the form
\begin{equation}
{\hat{\beta}}^2={1\over{K^2}} \sum_{k=0}^{K-1} || {\hat{\bfS}}_k-{\hat{\bfS}} ||^2_{\rm F},
\label{eq:eqnSHRmm}
\end{equation}
so the estimator of $\rho_0$ becomes 
\begin{equation}
{\hat{\rho}}_0=
{ {{\hat{\beta}}^2}  \over {{\hat{\delta}}^2} }=
{
{    \sum_{k=0}^{K-1}|| {\hat{\bfS}}_k-{\hat{\bfS}} ||^2_{\rm F}    }
 \over 
{  K^2 
\left[\tr\{ {\hat{\bfS}}^2\}-
({{ \tr^2\{ {\hat{\bfS}}\} }/ p})
\right] 
}
}\EqualDef {\hrhoLW},
\label{eq:eqnSHRtt}
\end{equation}
where we have defined this estimator to be ${\hrhoLW}$ because
it is of the same form as derived in \cite[pp.~379--380]{LedoitWolf04} for real-valued covariance matrices. 

Finally then the proposed shrinkage estimator of the spectrum  is, from (\ref{eq:eqnSHRff}), given by
\begin{equation}
\hSLW=\left[1-
\hrhoLW\right] 
{\hat{\bfS}}+
\hrhoLW{\hatloading}\bfI_p,
\label{eq:eqnSHRkk}
\end{equation}
exactly mimicking \cite[p.~380]{LedoitWolf04}.  As a result the empirical shrinkage of the eigenvalues is given by
$ 
{\hat\lambda}_i \rightarrow (1-\hrhoLW) {\hat\lambda}_i+\hrhoLW{\hatloading}.
$ 
This approach can be used if ${\hat{\bfS}}$ is singular or ill-conditioned.
Notice that if $K<p,$ so that ${\hat{\bfS}}$ is singular, the resulting zero eigenvalues will be modified to $\hrhoLW{\hatloading}.$ 

Note that since $\delta^2=\alpha^2+\beta^2$ if we define $ {\bar{\beta}}^2=\min\{ {\hat{\beta}}^2, {\hat{\delta}}^2 \}$ then
$
{ {{\bar{\beta}}^2}  / {{\hat{\delta}}^2} }= \min\{\hrhoLW,1\}
$
provides an estimate for the shrinkage parameter which is constrained by its theoretical upper bound of unity. This would be used in practical applications.

\begin{remark}
The form of ${\hat{\beta}}^2$ given in (\ref{eq:eqnSHRmm}) for the multitaper approach is very appealing as the averaging is all carried out at the frequency of interest, and is done over tapers. In the approach of \cite[p.~921]{BohmvSachs09} the ``local variance'' averaging must be done over different  frequencies.
\end{remark}

%%%%%%%%%%%%%%%%%%%%%%%%%%%%%%%%%%%%%%%%%%%%%%%%%%%%%%%%%%%%%%%%%%%%%%%%%%%%%%%%%%%%%%%%%%%%%%%
\begin{figure}[t]
\begin{center}
\includegraphics[height=1.75in,width=3.25in]{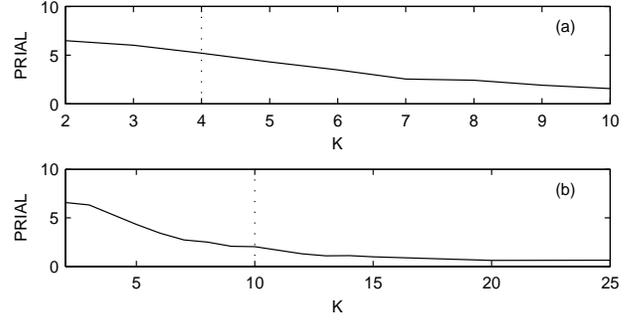}
\end{center}
\caption{ Simulated PRIAL values for (a) $\bfS_A$ for which $p=4$ and (b) $\bfS_B$ for which $p=10.$ In each case the dotted line indicates $p.$
\label{fig:PRIALS}}
\end{figure} 
%%%%%%%%%%%%%%%%%%%%%%%%%%%%%%%%%%%%%%%%%%%%%%%%%%%%%%%%%%%%%%%%%%%%%%%%%%%%%%%%%%%%%%%%%%%%%%%

\section{Rao-Blackwell Estimation}\label{sec:RaoB}
It is possible to produce another estimator from  $\hSLW$ which is at least as good under any convex loss criterion. The transformed estimator to be derived is known as the 
Rao-Blackwell estimator
and was developed for real-valued covariance matrices in the context of (\ref{eq:eqnSHRkk}) by \cite{Chen_etal10}.
The idea is that if $T(\bfJ_0,\ldots,\bfJ_{K-1})$ is a sufficient statistic
 for $\bfS,$ and if ${\cal S}(\bfJ_0,\ldots,\bfJ_{K-1})$ is an estimator for $\bfS,$ then the conditional expectation ${\cal S}'(\bfJ_0,\ldots,\bfJ_{K-1}){\displaystyle{\EqualDef}} E\{ {\cal S}(\bfJ_0,\ldots,\bfJ_{K-1})|T\}$
is never worse than ${\cal S}(\bfJ_0,\ldots,\bfJ_{K-1})$ under any convex loss criterion.  To see this, start with the risk
$R(\bfS,{\cal S})$ of the original estimator \cite[p.~483]{CasellaBerger90}
\begin{eqnarray}
R(\bfS,{\cal S})
&=& E_{\bfS}\{ L(\bfS,{\cal S}(\bfJ_0,\ldots,\bfJ_{K-1}))\}
\label{eq:eqnSHRai}
\\
&=&E_{\bfS}\{ E\{  L(\bfS,{\cal S}(\bfJ_0,\ldots,\bfJ_{K-1}))|T\}\}\nonumber\\
&\geq &
E_{\bfS}\{L(\bfS, E\{  {\cal S}(\bfJ_0,\ldots,\bfJ_{K-1})|T\})\}\nonumber\\
&=&E_{\bfS}\{L(\bfS, {\cal S}'(\bfJ_0,\ldots,\bfJ_{K-1}))\}
\label{eq:eqnSHRaj}\\
&=&R(\bfS,{\cal S}').\nonumber
\end{eqnarray}
(Here the second line uses the rule of iterated expectation
and the third line follows from Jensen's inequality and the assumed convexity of the loss function.)

In the context of spectral matrix estimation we note that under the independent complex Gaussian assumption for the $\bfJ_0,\ldots,\bfJ_{K-1},$ (\ref{eq:eqnSHRee}), that $\hbS$ is  a sufficient statistic for estimating $\bfS,$ \cite[Theorem~4.2]{Goodman63}; this is true for $K\geq p$ and $K<p.$ Then, the Rao-Blackwell estimator takes the form
$\hSRB = E\{ \hSLW | \hbS\}$ and 
\begin{eqnarray*}
R(\bfS,\hSLW)&=&E_{\bfS}\{ ||\hSLW-\bfS||_{\rm F}^2\}
%\label{eq:eqnSHRai}
\\
&=&E_{\bfS}\{ E\{  ||\hSLW-\bfS||_{\rm F}^2|\hbS\}\}\nonumber\\
&\geq &
E_{\bfS}\{||E\{ \hSLW | \hbS\} -\bfS||_{\rm F}^2 \}\nonumber\\
&=&E_{\bfS}\{ ||\hSRB-\bfS||_{\rm F}^2\}
=R(\bfS,\hSRB).
%\label{eq:eqnSHRaj}\\
\end{eqnarray*}
So,
\begin{eqnarray*}
\hSRB &=& E\{ \hSLW | \hbS\}=
E\{ \left[1-
\hrhoLW\right] 
{\hat{\bfS}}+
\hrhoLW{\hatloading}\bfI_p | {\hbS}\}\\
&=&[1-E\{ \hrhoLW|{\hbS}\}]{\hbS}+E\{\hrhoLW {\hatloading}|{\hbS}\}\bfI_p\\
&{\displaystyle{\EqualDef}} &[1-\hrhoRB]{\hbS}+\hrhoRB{\hatloading}\bfI_p,
\end{eqnarray*}
where the Rao-Blackwell shrinkage parameter $\hrhoRB$ is 
\begin{equation}
\hrhoRB \EqualDef E\{ \hrhoLW|{\hbS}\}=
{
{ E\left\{   \sum_{k=0}^{K-1}|| {\hat{\bfS}}_k-{\hat{\bfS}} ||^2_{\rm F}\,\, | \hbS\right\}   }
 \over 
{  K^2 
\left[\tr\{ {\hat{\bfS}}^2\}-
({{ \tr^2\{ {\hat{\bfS}}\} }/ p})
\right] 
}
}.\label{eq:formRB}
\end{equation}

The form of the shrinkage parameter was derived in  \cite{Chen_etal10} for real-valued covariance matrices.
For our complex-valued case the form is substantially different. 
\begin{theorem}\label{theorem:three}
Under the assumption (\ref{eq:eqnSHRee}), $\hrhoRB$ in (\ref{eq:formRB}) takes the simple form
\begin{equation}
\hrhoRB ={
{
\tr^2\{{\hbS}\}- (\tr\{ {\hbS}^2\}/K) 
}
\over
{
(K+1)\left[ \tr\{{\hbS}^2\}-(\tr^2\{{\hbS}\}/p)\right]
}
}.
\label{eq:eqnSHRah}
\end{equation}
\end{theorem}
\begin{proof}
This uses invariance properties of the random matrix $\bfJ$ and the random unitary matrices arising from its singular value decomposition. Details are given in Appendix\ref{app:th3}:
put the results of Lemma~\ref{lemma:three} and Lemma~\ref{lemma:four}  into the numerator of (\ref{eq:formRB}), then (\ref{eq:eqnSHRah}) readily follows.
\end{proof}

From (\ref{eq:eqnSHRai}) and (\ref{eq:eqnSHRaj}) we have that
$
E_{\bfS}\{ \|\hSLW-\bfS\|^2_{\rm F}\}\geq
E_{\bfS}\{ \|\hSRB-\bfS\|^2_{\rm F}\}.
$ 
It is common to look at such a difference via the percentage relative improvement in average loss
(PRIAL) defined as
\begin{equation*}%\label{eq:PRIALdef}
{\rm PRIAL}\EqualDef 100
\frac{
E_{\bfS}\{ \|\hSLW-\bfS\|^2_{\rm F}\}-
E_{\bfS}\{ \|\hSRB-\bfS\|^2_{\rm F}\}
}
{
E_{\bfS}\{ \|\hSLW-\bfS\|^2_{\rm F}\}
} .
\end{equation*}
To illustrate this quantity two different  Hermitian matrices, $\bfS_A$
and $\bfS_B$ were utilized. 
$\bfS_A$ is the $4\times 4$ `random' choice
$$
\bfS_A=
\left[
\begin{matrix}
10 & 7+\eye & 8 &  4 \\
7-\eye & 12 & 6+2\eye & 5-\eye\\
8 & 6-2\eye & 15 & 9-3\eye \\
4 & 5+\eye & 9+3\eye & 10
\end{matrix}
\right]
$$
and the second $\bfS_B$ is set equal to a $10\times 10$ estimated spectral matrix  from an EEG dataset.
 From each of these $\bfS$ matrices,
a set of $m=5000$ matrix estimates $\hbS_1,\ldots,\hbS_m$ were simulated satisfying (\ref{eq:eqnSHRdd}) and (\ref{eq:eqnSHRee}).
For each replication, estimates were constructed of the form  $\hSLW$ and $\hSRB,$ and the Frobenius norm between the estimate and the true matrix ($\bfS_A$ or $\bfS_B$) was found. The results were averaged over the 5000 replications to give estimates of 
$E_{\bfS}\{ \|\hSLW-\bfS\|^2_{\rm F}\}$ and 
$E_{\bfS}\{ \|\hSRB-\bfS\|^2_{\rm F}\}.$
This was done for $K<p$ (singular case) and $K\geq p$ (non-singular). The results are shown in Fig.~\ref{fig:PRIALS}. 
Behaviour seems quite smooth as $K$ crosses from the
singular to non-singular cases.
The Rao-Blackwell estimator offers a useful improvement over the Ledoit-Wolf estimator. In these examples  the PRIAL decreases almost monotonically with increasing degrees of freedom, $K,$ but this behaviour need not hold for other choices for $\bfS.$

Note that, analogously to the Ledoit-Wolf estimate of the shrinkage parameter, $\min\{ \hrhoRB,1\}$
provides an estimate for the shrinkage parameter which is constrained by its theoretical upper bound of unity, and would be used in practice.

\begin{remark}
In \cite{Chen_etal10} an oracle approximating shrinkage (OAS)  estimator was given. The analogous estimator in the complex case for (\ref{eq:eqnSHRae}) was found to be unpredictable. For example, for $\bfS_A$  while for $K=2$ the PRIAL (comparing to the Ledoit-Wolf estimator) was increased from 6.5\% (Rao-Blackwell) to 15\% (OAS), for $K=4$ it decreased from 5.2\% (Rao-Blackwell) to 1.0\% (OAS). The behaviour of the Rao-Blackwell estimator seems better suited for practical use.
It should also be pointed out that the oracle in (\ref{eq:eqnSHRae}) is optimal for the stochastic target, while $\hrhoLW$ and $\hrhoRB$ were developed for the deterministic target optimization.
\end{remark}
%%%%%%%%%%%%%%%%%%%%%%%%%%%%%%%%%%%%%%%%%%%%%%%%%%%%%%%%%%%%%%%%%%%%%%%%%%%%%%%%%%%%%%%%%%%%%%%
\begin{figure}[t]
\begin{center}
\includegraphics[height=2.75in,width=3.25in]{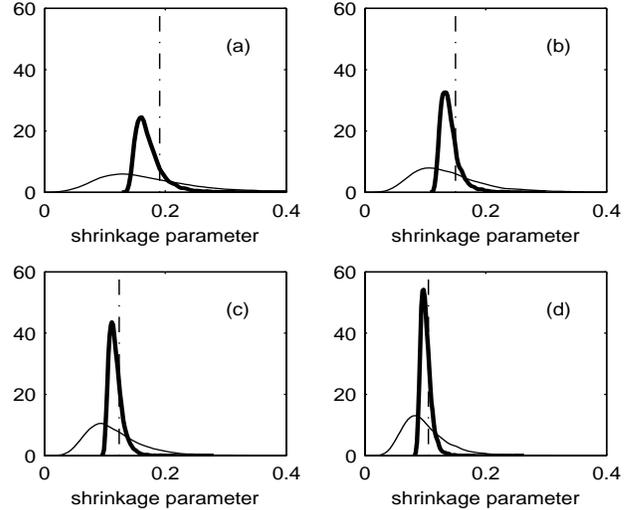}
\end{center}
\caption{Simulated distributions for $\hrhoLW$ (thin line) and 
$\hrhoRB$ (thick line) for the $10\times 10$ matrix $\bfS_B$ for (a) $K=6,$ (b) $K=8,$ (c) $K=10$ and (d) $K=12.$ The vertical dash-dot line shows the oracle solution $\rho_0$ of (\ref{eq:eqnSHRext}).
\label{fig:shrinkp}}
\end{figure} 
%%%%%%%%%%%%%%%%%%%%%%%%%%%%%%%%%%%%%%%%%%%%%%%%%%%%%%%%%%%%%%%%%%%%%%%%%%%%%%%%%%%%%%%%%%%%%%%

Fig.~\ref{fig:shrinkp} compares the empirical distributions of $\hrhoLW$ and 
$\hrhoRB$ for the matrix $\bfS_B$ $(p=10)$ for (a) $K=6,$ (b) $K=8,$ (c) $K=10$ and (d) $K=12.$ As expected as $K$ increases, $\hrhoLW$ and 
$\hrhoRB$ reduce in variance and converge toward the oracle solution. The distribution of $\hrhoRB$ is always preferable to that of $\hrhoLW.$

In the rest of the paper we turn our attention to estimation of inverse spectral matrices.

\section{Rao-Blackwell Estimation for Inverse Spectral Matrices}\label{sec:invRaoB}

We denote the inverse of the spectral matrix, i.e., the precision matrix, by
$\bfC\, {\displaystyle{\EqualDef}}\, \bfS^{-1}.$  We shall firstly show that $\hSRB^{-1}$  is actually a ``Rao-Blackwellized'' estimator for $\bfC.$

\begin{lemma}
The inverse, $\hSRB^{-1},$ of the Rao-Blackwell estimator, $\hSRB,$ is in the form of a ``Rao-Blackwellized'' estimator for $\bfC.$
\end{lemma}
\begin{proof}
Firstly we note that $\hbS$ is a sufficient statistic for $\bfC.$ To see this we note that the  probability density function for $\bfJ_0,\ldots,\bfJ_{K-1}$ can be written
$$
p(\bfJ_0,\ldots,\bfJ_{K-1}; \bfC)= \pi^{-pK} {\rm det}^K\{\bfC\}
\exp[-K\tr\{ \bfC \hbS \}].
$$
The part that depends on $\bfC$ only depends on the sample through $\hbS,$ so this is a sufficient statistic for $\bfC$ by the factorization theorem \cite{HalmosSavage49}. Now
$
\hSRB(\hbS)= E\{ \hSLW | \hbS\}
$
is an estimator for $\bfS,$ so $\hSRB^{-1}(\hbS)$ is an estimator for $\bfC$. Recall the general result that for a function $h(\cdot),$
$$
E\{ h(\hbS)| \hbS\}=h(\hbS),
$$
so 
$$
E\{ \hSRB^{-1}(\hbS)| \hbS\}= \hSRB^{-1}(\hbS)\EqualDef {\hat \bfC}_{\rm RB}(\hbS),
$$
which completes the proof.
\end{proof}
Clearly we can use ${\hat \bfC}_{\rm RB}(\hbS)$ to estimate $\bfC$ when $\hbS$ is singular, $K<p,$ or non-singular, $K\geq p.$
% RB model is Rob W's
% Phi=[0.2 0 -0.1 0 -0.5;...
%     0.4 -0.2 0 0.2 0;...
%     -0.2 0 0.3 0 0.1;...
%     0.3 0.1 0 0.3 0;...
%     0 0 0 0.5 0.2]; 
 
In order to illustrate the Rao-Blackwellized estimator for $\bfC$ a stable and stationary vector autoregressive 
process of order 1 and dimension $p=5$ 
(VAR$_5(1)$) was utilized. The process was simulated 5000 times with $N=1000$ and $K=4.$
Fig.~\ref{fig:SinverseRB} shows the resulting (estimated) PRIAL 
\begin{equation}\label{eq:PRIAL_RB}
{\rm PRIAL}\EqualDef 100
\frac{
E_{\bfS}\{ \|{\hat \bfC}_{\rm LW}-\bfC\|^2_{\rm F}\}-
E_{\bfS}\{ \|{\hat \bfC}_{\rm RB}-\bfC\|^2_{\rm F}\}
}
{
E_{\bfS}\{ \|{\hat \bfC}_{\rm LW}-\bfC\|^2_{\rm F}\}
},
\end{equation}
where ${\hat \bfC}_{\rm LW}=\hSLW^{-1}.$
The PRIAL reaches as much as   15\% for some frequencies showing that the Rao-Blackwell approach can be a worthwhile improvement over the Ledoit-Wolf estimator even for dimension $p=5.$
%Fig.~\ref{fig:Sinverse}(b) shows the PRIAL %when 
%the calculable ${\hat \bfC}_{\rm RB}$ is %replaced by the unknown estimator
%%%%%%%%%%%%%%%%%%%%%%%%%%%%%%%%%%%%%%%%%%%%%%%%
\begin{figure}[t]
\begin{center}
\includegraphics
[height=2in,width=3in]{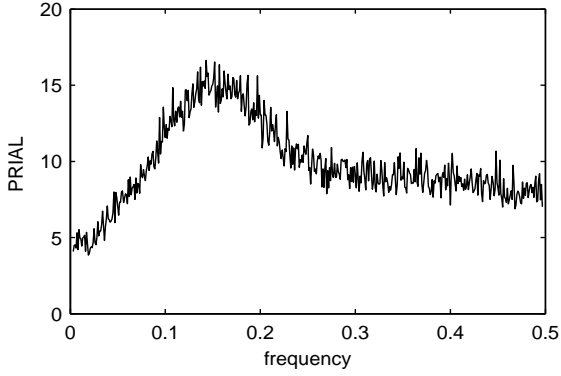}
%{Sinverse_RaoB_OracleDet.eps}
\end{center}
\caption{Estimated PRIAL (\%) (improvement of ${\hat \bfC}_{\rm RB}$ over
${\hat \bfC}_{\rm LW}$) for a VAR$_5(1)$ time series example.
\label{fig:SinverseRB}}
\end{figure} 
%%%%%%%%%%%%%%%%%%%%%%%%%%%%%%%%%%%%%%%%%%%%%%%

\section{Random Matrix Approach to Inverse Spectral Matrices}\label{sec:Marzetta}
Marzetta {\it et al.} \cite{Marzetta_etal11} examined how to manipulate a singular ($K < p$) covariance matrix constructed from circularly-symmetric complex vectors to obtain a non-singular version. 
In the context of spectral matrices, we can explain their idea as follows.

Firstly an ensemble of $L\times p$ {\it random}\/  matrices $\bfPhi\in{\mathbb{C}}^{L\times p},$ with $L\leq K <p,$ is introduced, 
which have orthonormal rows, so that
$\bfPhi\bfPhi^H=\bfI_L.$ Such matrices are often called `semi-unitary' and 
were chosen to be bi-unitarily invariant
(see Appendix\ref{app:invariance}). Such matrices are  called ``isotropically random''  with the Haar distribution in \cite{Marzetta_etal11}.

The $L\times L$ matrix  $\bfPhi{\hat {\bfS}}\bfPhi^H$  is invertible (with probability one). 
\cite{Marzetta_etal11} advocate
inverting this matrix and   projecting out the result to a $p\times p$ matrix again using the random semi-unitary matrix
$\bfPhi.$ Then taking the conditional expectation over the semi-unitary ensemble, 
 gives
$$
{\hat {\bfC}}^{\star}_L( {\hat {\bfS}}) \EqualDef
(p/L)\,E_{\bfPhi}\{ \bfPhi^H[\bfPhi{\hat {\bfS}}\bfPhi^H]^{-1}\bfPhi\,\big |\,{\hat {\bfS}}\},
%\eqno\EqFigDef\eqnSHRxx
$$
as an estimator for $\bfC.$ Although not given explicitly in \cite{Marzetta_etal11} a rescaling by $(p/L)$ has been included as in \cite{TucciWang12} so that the estimate of the inverse of the identity matrix is the identity. The term $L$ such that $L<K<p$ is a parameter to be chosen; its determination is discussed later.

Since here $K<p,$ the Hermitian matrix ${\hat\bfS}$ has rank $r=\min\{p,K\}=K$ with probability 1. Its spectral decomposition is ${\hat{\bfS}}={\bfU} {\bfLambda} {{\bfU}^H},$ where 
$${\bfLambda}={\rm diag}\{ {\lambda}_1,\ldots,{\lambda}_K,\underbrace{\;0, \ldots, 0   }_{p-K\rm\;times}\}
$$  is the diagonal matrix of estimated eigenvalues, (ordered largest to smallest), and ${{\bfU}}$ is the unitary matrix having corresponding eigenvectors for its columns. From \cite{Marzetta_etal11} it follows that
\begin{equation}
{\hat {\bfC}}^{\star}_L( {\hat {\bfS}})=(p/L){{\bfU}}\, {\hat {\bfC}}^{\star}_L( { {\bfLambda}})\, {{\bfU}}^H,
\label{eq:eqnSHRad}
\end{equation}
so the required estimator can be constructed from ${\hat {\bfC}}^{\star}_L( {{\bfLambda}}).$ Further, \cite{Marzetta_etal11} show that
\begin{equation}\label{eq:further}
{\hat {\bfC}}^{\star}_L(  {\bfLambda})={\rm diag}\{ {\lambda}^{\star}_1,\ldots,{\lambda}^{\star}_K,{\lambda}^{\star}, \ldots, {\lambda}^{\star}   \},
\end{equation}
where ${\lambda}^{\star}_i,\,i=1,\ldots,K$ are modified versions of 
 ${\lambda}_i,\,i=1,\ldots,K,$ and the $p-K$ zero eigenvalues of ${\hat{\bfS}}$  have been replaced by $p-K$ copies of a single value, ${\lambda}^{\star}.$

\subsection{Computations via simulations}
The computation of ${{\lambda}}^{\star}_i,\,i=1,\ldots,K$ and ${{\lambda}}^{\star}$ can be carried out purely via simulation, as done by
\cite{Marzetta_etal11} 
(personal correspondence with Gabriel Tucci). However,
for a given ${\hat{\bfS}}$, in order to get good agreement between the estimator of $\bfS$  derived by averaging many copies of $\bfPhi^H[\bfPhi{{\bfLambda}}\bfPhi^H]^{-1}\bfPhi$ for different $\bfPhi,$ (followed by premultiplication by ${{\bfU}}$ and post-multiplication by ${{\bfU}}^H$), and the analytic estimator to be described below, the number of copies needing to be averaged is typically very large. For example the order of $10^6$ $\bfPhi$'s were required for the $p=10$ channel EEG example to achieve agreement to two significant figures. The corresponding compute-time cost turned out to be around 5000 times as heavy, about 500s for the simulation approach versus 0.1s for the analytic scheme at any frequency.
 Even with modern computational power this sort of simulation burden is not suitable in a spectral matrix context where 
$\bfC$  must be estimated at possibly thousands of  frequencies.

\subsection{Computations using analytic methods}
We now examine how to  compute (\ref{eq:further})  using analytic methods.
Define $\bfD_K={\rm diag}\{{{\lambda}}_1,\ldots,{{\lambda}}_K\}.$
Then \cite[Theorem 1]{Marzetta_etal11}, for a continuous function $g(\cdot),$
\begin{equation}
\int_{\Omega_{0}}\frac{1}{K} \tr\{ g(\bfPhi^H_0 \bfD_K\bfPhi_0)\}\dif\bfPhi_0
=\!
\sum_{k=0}^{L-1} \frac
{(K-(k+1))!{\det\{ \bfG_k\}} }
{(L-(k+1))!{\det\{ \bfV_K \}} }
\label{eq:eqnSHRyy}
\end{equation}
Here $
{\Omega_{0}}{\displaystyle{\EqualDef}}  
\{\bfPhi_0\in{\mathbb{C}}^{K\times L}:\bfPhi_0^H\bfPhi_0=\bfI_L\},$ these  matrices with orthonormal columns again being bi-unitarily invariant (Haar distributed) --- see Lemma~\ref{Def:nonsquare} of Appendix\ref{app:invariance}.
$\bfV_K$ is the Vandermonde matrix associated with $\bfD_K$ given in the `flipped' form
$$
\bfV_K=
\left[
\begin{matrix}
{{\lambda}}_1^{K-1} & {{\lambda}}_2^{K-1} & \cdots & {{\lambda}}_K^{K-1}\\
{{\lambda}}_1^{K-2} & {{\lambda}}_2^{K-2} & \cdots & {{\lambda}}_K^{K-2}\\
\vdots & \vdots & \vdots & \vdots\\
{{\lambda}}_1 & {{\lambda}}_2 & \cdots & {{\lambda}}_K\\
1 & 1 & \cdots & 1
\end{matrix}
\right],
$$
and $\bfG_k$ is the matrix defined by replacing row $(k+1)$ of the Vandermonde matrix $\bfV_K,$ namely $[  {{\lambda}}_1^{K-(k+1)},\ldots
{{\lambda}}_K^{K-(k+1)}],$ by the row
\begin{eqnarray}
&&\!\!\!\!\!\!\!\!\!\!\!\!\!\!\!\!\!\!\!\!\!\!\!\!
\left[ I^{(K-L)}\{x^{L-(k+1)} g(x) \}\Big|_{x={{\lambda}}_1},\ldots,\right.\nonumber\\
&&
\left. I^{(K-L)}\{x^{L-(k+1)} g(x) \}\Big|_{x={{\lambda}}_K} \right],
\label{eq:eqnSHRzz}
\end{eqnarray}
where $I^{(q)}\{ f(x)\}$ denotes $q$ integrations of $f(x).$ 

We consider first the computation of  ${{\lambda}}^{\star}_i,$ for which \cite[p.~6265]{Marzetta_etal11}
\begin{equation}
{{\lambda}}^{\star}_i={
{\partial}
\over 
{\partial {{\lambda}}_i}
}
\int_{\Omega_{0}}\frac{1}{K} \tr\{ \log(\bfPhi_0^H \bfD_K\bfPhi_0)\}\dif\bfPhi_0.
\label{eq:eqnSHRac}
\end{equation}
The integral component is given by (\ref{eq:eqnSHRyy})\/ with $g(\cdot)\equiv\log(\cdot).$ So to compute $\bfG_k$ via (\ref{eq:eqnSHRzz}) we need to know terms like
$I^{(q)}\{ x^n \log x\}$ for $q\geq 1, n \geq 0.$ This is found to be,
$$
I^{(q)}\{ x^n \log x\}=
{
{x^{n+q} n!}
\over
{(n+q)!}
}
\left[\log x-\sum_{j=1}^q {1\over{n+j}}\right].
$$
To calculate ${{\lambda}}^{\star}_i$ in (\ref{eq:eqnSHRac}) we can now use (\ref{eq:eqnSHRyy}),
\begin{eqnarray*}
{{\lambda}}^{\star}_i
&=&
\sum_{k=0}^{L-1} 
\frac
{(K-(k+1))! }
{(L-(k+1))! }
\frac
{\partial}
{\partial {{\lambda}}_i}
\left[
\frac
{\det\{ \bfG_k\}}
{\det\{ \bfV_K \}}
\right].
\end{eqnarray*}

The partial derivative on the right is given by
$$
\frac
{
{\det\{ \bfV_K\}}
\frac
{\partial}
{\partial {{\lambda}}_i}
{\det\{ \bfG_k\}}
-
{\det\{ \bfG_k\}}
\frac
{\partial}
{\partial {{\lambda}}_i}
{\det\{ \bfV_K\}}
}
{
{\det^2\{ \bfV_K \}}
}.
$$
To find the derivative of the determinant of a $K\times K$ matrix $\bfM$ ($\bfG_k$ or $\bfV_K$) we first differentiate all entries of the matrix $\bfM$ 
by ${{\lambda}}_i;$ denote the $(l,m)$th resulting entry  by $A_{l,m}.$
Now let $\bfB$ be the cofactor matrix corresponding to
$\bfM.$ For $1\leq l,m\leq K$ define
$
D_{l,m}= A_{l,m} B_{l,m},
$
the element-by-element multiplication of the matrices $\bfA$ and $\bfB.$ Then the derivative of the determinant is given by \cite[eqn.~6]{Golberg72}
$$
\frac 
{\partial} 
{\partial {{\lambda}}_i}
\det\{ \bfM\}=\sum_{l,m=1}^K D_{l,m}.
$$
For the matrix $\bfV_K$,
$$
A_{l,m}=
\begin{cases}
(K-l){{\lambda}}_i^{K-(l+1)},& {\rm if }\,\, m=i;\\
0,& {\rm otherwise}.
\end{cases}
$$
For $\bfG_k,$ entry $A_{l,m}$ is given by
$$
\begin{cases}
(K-l){{\lambda}}_i^{K-(l+1)},& {\rm if }\,m=i, l\not=k+1;\\
\frac
{\partial} 
{\partial {{\lambda}}_i}
I^{(K-L)}\{x^{L-(k+1)} \log(x) \}\Big|_{x={{\lambda}}_i},& {\rm if }\,m=i, l=k+1;\\
0,& {\rm otherwise},
\end{cases}
$$
where of course we can simplify the second term to
$$
I^{(K-L-1)}\{x^{L-(k+1)} \log(x) \}\Big|_{x={{\lambda}}_i}.
$$
The cofactor matrices for $\bfG_k$ or $\bfV_K$ can be readily found using standard matrix software. Hence we are able to compute ${{\lambda}}^{\star}_i, \,i=1,\ldots,K.$

The computation of ${{\lambda}}^{\star}$ is straightforward. We know
\cite[p.~6264]{Marzetta_etal11}
that for $L<K,$
$
{{\lambda}}^{\star}={
{\det\{ \bfG\}}
/
{\det\{ \bfV_K \}}
}
$
with 
$\bfG$ being the matrix defined by replacing the $L$th row of the Vandermonde matrix $\bfV_K,$ namely $[  {{\lambda}}_1^{K-L},\ldots,
{{\lambda}}_K^{K-L}],$ by the row 
$
\left[
{{\lambda}}_1^{K-(L+1)} \log {{\lambda}}_1,\ldots,
{{\lambda}}_K^{K-(L+1)} \log {{\lambda}}_K\right].
$
We are thus  able to compute all the components of
(\ref{eq:further}) 
 and therefore 
${\hat {\bfC}}^{\star}_L( {\hat {\bfS}})$ in 
(\ref{eq:eqnSHRad}).

\subsection{Choice of $L$}
In practice we must choose a suitable value  of $L$ to use. Use of the analytic results  means we require $L <K$ and we are interested in the singular case $K<p.$ To select $L$ we proceed by seeking $L={\hat L}$ that minimizes the predictive risk
defined as
$$
{\rm PR}(\ell)=E\left\{ E_{{\tilde{\bfJ}}} \{ \|{\hat{\bfC}}^\star_\ell
{\tilde{\bfJ}}{\tilde{\bfJ}}^H-\bfI_p\|^2_{\rm F} \big| \bfJ_0,\ldots,\bfJ_{K-1} \}\right\}
$$
where ${\hat{\bfC}}^\star_\ell$ is the estimated inverse spectral matrix found from $\bfJ_0,\ldots,\bfJ_{K-1}$ when $L=\ell,$ and ${\tilde{\bfJ}}$ is independent of the $\bfJ_k$'s and from the same distribution. Here we have used quadratic loss which does not involve any further matrix inversions. We approximate the predictive risk using leave-one-out cross-validation. Specifically,
the estimate of the predictive risk is
$$
{\widehat{\rm PR}}(\ell)={1\over K} \sum_{j=1}^K\| 
{\hat{\bfC}}^{\star[j]}_\ell\bfJ_j\bfJ_j^H-\bfI_p\|^2_{\rm F},
$$
where ${\hat{\bfC}}^{\star[j]}_\ell$ denotes the estimated inverse spectral matrix found from $\bfJ_0,\ldots,\bfJ_{K-1}$ excluding $\bfJ_j.$ Then we take
\begin{equation}\label{eq:chooseL}
{\hat{L}}=\arg\min_{\ell} {\widehat{\rm PR}}(\ell).
\end{equation}
Note that using this scheme it is only possible to consider values of $\ell <K-1$ since we know that ordinarily $L$  must be less than $K$ but additionally here ${\hat{\bfC}}^{[j]}_\ell$ is derived from $K-1$ of the $\bfJ_j$'s.

\subsection{Example}
In order to illustrate the random matrix  estimator ${\hat {\bfC}}^{\star}_L( {\hat {\bfS}})$ for $\bfC$ in a time series context, a stable and stationary vector autoregressive 
process of order 1 and dimension $p=10$ 
(VAR$_{10}(1)$) was utilized with $N=1000$ and $K=8.$ At each frequency (\ref{eq:chooseL}) was used to choose $L.$ 
Fig.~\ref{fig:SinverseMarzetta} shows the resulting (estimated) PRIAL 
\begin{equation}\label{eq:PRIAL_Marzetta}
{\rm PRIAL}\EqualDef 100
\frac{
E_{\bfS}\{ \|{\hat \bfC}_{\rm LW}-\bfC\|^2_{\rm F}\}-
E_{\bfS}\{ \|{\hat {\bfC}}^{\star}_L-\bfC\|^2_{\rm F}\}
}
{
E_{\bfS}\{ \|{\hat \bfC}_{\rm LW}-\bfC\|^2_{\rm F}\}
}.
\end{equation}
This estimated PRIAL was found from 100 replications and because of the need to produce the replications computations were carried out only at every 10th Fourier frequency.
The PRIAL reaches nearly  20\% for some frequencies again showing  a worthwhile improvement over the Ledoit-Wolf estimator.

%%%%%%%%%%%%%%%%%%%%%%%%%%%%%%%%%%%%%%%%%%%%%%%%
\begin{figure}[t]
\begin{center}
\includegraphics[height=1.75in,width=3in]{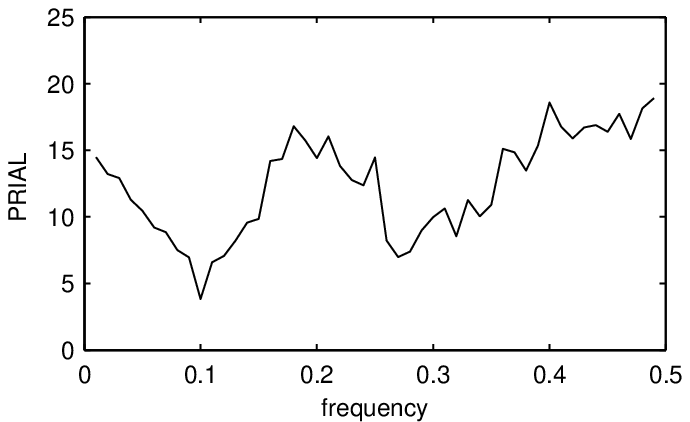}
\end{center}
\caption{Estimated PRIAL (\%) (improvement of ${\hat {\bfC}}^{\star}_L$ over
${\hat \bfC}_{\rm LW}$) for a VAR$_{10}(1)$ time series example.
\label{fig:SinverseMarzetta}}
\end{figure} 
%%%%%%%%%%%%%%%%%%%%%%%%%%%%%%%%%%%%%%%%%%%%%%%

\section{Application to EEG data}\label{eq:EEGapp}
We now  compute ${\hat \bfC}_{\rm RB}$ and
${\hat {\bfC}}^{\star}_L$ for 
electroencephalogram (EEG) data, (resting conditions with eyes closed), for a patient
diagnosed with positive syndrome schizophrenia.
Interest was in the delta frequency range, $0.5< f \leq 4$Hz, see \cite{Medkouretal10}. EEG was recorded on the scalp  at $10$  sites, so $\{\bfX_t\}$ is a $p=10$ vector-valued process,
using a bandpass filter of 0.5--45Hz and sample interval of $\deltt=0.01$s. To remove the dominant and contaminating 10Hz alpha rhythm, which would otherwise cause severe spectral leakage, the data was low-pass filtered and resampled to a sample interval of $\deltt=0.05$s. After this downsampling $N=612$. 

Using this real data the spectral matrix $\bfS(f)$ was estimated as ${\bfS}_0(f),$ say, for $|f|\leq \fN$ , using $K=40$ tapers.
Using the vector-valued circulant embedding approach, 
\cite{ChandnaWalden13},
100 independent Gaussian $p$-vector-valued time series ($p=10$) were computed, each having  $\bfS_0(f),\, |f|\leq \fN,$ as its true spectral matrix. 
For each of these time series the singular matrix $\hbS(f)$  was computed using multitaper estimation with $K=8$ tapers for 100 frequencies equally spaced between 0.5 and 4Hz, and from these estimates
${\hat \bfC}_{\rm RB}$ and
${\hat {\bfC}}^{\star}_L$ were computed,
(with (\ref{eq:chooseL}) choosing $L$ for  ${\hat {\bfC}}^{\star}_L$).
The estimated PRIAL --- with $\bfC=\bfS_0^{-1}$ --- was then found over the 100 replications. In this way the simulation experiment mimicks the spectral properties of the EEG data while providing calibrated results, which are shown in Fig.~\ref{fig:EEGexample}. We see that both schemes improve on the LW method,
but that ${\hat {\bfC}}^{\star}_L$ does particularly well, with PRIAL reaching 50\%.

%%%%%%%%%%%%%%%%%%%%%%%%%%%%%%%%%%%%%%%%%%%%%%%%
\begin{figure}[t]
\begin{center}
\includegraphics[height=2in,width=3in]{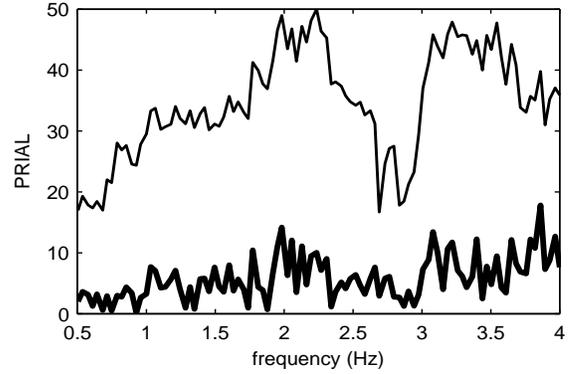}
\end{center}
\caption{Estimated PRIAL (\%) for EEG data. 
Improvement of ${\hat \bfC}_{\rm RB}$ over
${\hat \bfC}_{\rm LW}$ is shown by the thick line. Improvement of ${\hat {\bfC}}^{\star}_L$ over
${\hat \bfC}_{\rm LW}$ is shown by the thin line. 
\label{fig:EEGexample}}
\end{figure} 
%%%%%%%%%%%%%%%%%%%%%%%%%%%%%%%%%%%%%%%%%%%%%%%

\section{Concluding Discussion}
We have described two analytical estimators (Rao-Blackwell and random matrix)  for the spectral precision matrix.
Interestingly, ${\hat \bfC}_{\rm RB}$ is the inverse of a shrinkage estimator where the shrinkage parameter is obtained as a conditional expectation, conditional on $\hbS$, while the random matrix estimator 
${\hat {\bfC}}^{\star}_L$ is also a conditional expectation, again conditioned on $\hbS.$
We have shown that both 
hold promise for being useful in practice, offering possibly substantial improvements over the inverse of the LW estimator of $\bfC.$ Further investigation of their properties seems worthwhile. 

\appendix{}
To simplify notation we drop  explicit frequency dependence. 
\subsection{Bi-unitary invariance}\label{app:invariance}

\begin{definition}\label{Def:biunit}
A complex-valued  $n\times m$ random matrix $\bfZ$ is right(left)-unitarily invariant if its distribution is invariant under the transformation $\bfZ\rightarrow \bfZ\bfTheta$ 
($\bfZ\rightarrow \bfUpsilon\bfZ$)
where $\bfTheta\in {\cal U}(m), \bfUpsilon\in {\cal U}(n),$ where ${\cal U}(n)$ is the compact group of all $n\times n$ complex unitary matrices, i.e.,
${\cal U}(n)=\{ \bfU_{n\times n}: \bfU^H \bfU=\bfI_n\}.$ If both are true we say $\bfZ$ is bi-unitarily invariant.

\begin{lemma}\label{lemma:one}
The matrix $\bfJ$ defined in (\ref{eq:defJ}) with $\bfJ_k$ given by
(\ref{eq:eqnSHRee}) is right-unitarily invariant. (If $\bfS=\bfI_p$ it is bi-unitarily invariant.)�
\end{lemma}
\begin{proof}
This follows from \cite[p.~487]{James64}.
\end{proof}

\end{definition}
\begin{lemma}\label{def:two}
When considered as a metric space ${\cal U}(n)$ is measurable. 
There is a unique left-unitarily invariant probability measure $\mu$ for ${\cal U}(n)$ such that $\mu(\bfTheta \bfA)=\mu(\bfA)$ for any measurable $\bfA \subset {\cal U}(n)$ and any $\bfTheta\in {\cal U}(n).$ Moreover, since ${\cal U}(n)$ is compact, the same measure $\mu$ is also right-unitarily invariant. The Haar measure  is this unique probability measure $\mu$  on ${\cal U}(n)$ that is bi-unitarily invariant.
See 
\cite[p.~108]{TracyWidom93}.
%\cite{Halmos50}. 
\end{lemma}

\begin{remark}
Let $\bfUpsilon\in {\cal U}(n).$ If $\bfUpsilon$ has Haar measure then for all $\bfTheta_1, \bfTheta_2 \in {\cal U}(n),$ $p(\bfTheta_1\bfUpsilon\bfTheta_2)=p(\bfUpsilon),$ where $p(\bfUpsilon)$ denotes the joint probability density function of the components of the unitary matrix.
\end{remark}

\begin{lemma}\label{Def:nonsquare}
Let $\bfUpsilon\in {\cal U}(n)$ equipped with Haar measure. We now consider two specific truncations of the $n\times n$ unitary matrices.
Suppose we partition $\bfUpsilon$ in two ways:
$$
\bfUpsilon =
\left[
\begin{array}{c}
\bfPhi \\ \hline
\bfP_{(n-m)\times n}
\end{array}
\right]=
\left[
\begin{array}{c|c}
\bfPhi_0 & \bfQ_{n \times (n-m)} \\
\end{array}
\right],
$$
where $\bfPhi$ is $m\times n, m<n$ and $\bfPhi_0$ is $n\times m, m< n.$
Then $\bfUpsilon \rightarrow \bfPhi$ maps the unitary group onto the Stiefel manifold of $m\times n$ matrices with orthonormal rows, $\bfPhi\bfPhi^H=\bfI_m.$ The image of the Haar measure under this map is bi-unitarily invariant. Likewise, $\bfUpsilon \rightarrow \bfPhi_0$ maps the unitary group onto the Stiefel manifold of $n\times m$ matrices with orthonormal columns, $\bfPhi_0^H\bfPhi_0=\bfI_m.$ The image of the Haar measure under this map is again bi-unitarily invariant. See \cite{Fyodorov07}.
\end{lemma}

\subsection{Results required for proof of Theorem~\ref{theorem:three}}\label{app:th3}

%Our proof is of necessity generally very %different in form to the equivalent provided %in \cite{Chen_etal10} for the real-valued %case, since references used there are not %available in the complex-valued case. 

\begin{theorem}\label{theorem:main}\label{eq:proof}
We know that 
the singular value decomposition (SVD) for the $p \times K$ 
random matrix $\bfJ$ defined by (\ref{eq:defJ}) and (\ref{eq:eqnSHRee})  is \cite[p.~182]{Bernstein05}
$
\bfJ=\bfU\bfPsi\bfV^H,
$
where $\bfU\in {\cal U}(p), \bfV\in{\cal U}(K)$ and $\bfPsi$ is the $p\times K$ matrix
$$
\bfPsi=
\left[
\begin{array}{c|c}
\bfOmega & {\bf0}_{r \times (K-r)} \\ \hline
{\bf0}_{(p-r)\times r} & {\bf 0}_{(p-r)\times (K-r)}
\end{array}
\right],
$$
$\bfOmega$ is the diagonal matrix $\bfOmega={\rm diag}\{\omega_1,\ldots,\omega_r\},$ $\omega_i=\lambda_i^{1/2},$ the square root of the $i$th ordered eigenvalue $\lambda_i(\bfJ\bfJ^H)=\lambda_i(\bfJ^H\bfJ).$ Here $r={\rm rank}\{ \bfJ\}={\rm rank}\{\bfJ\bfJ^H\}={\rm rank} \{\bfJ^H\bfJ\}.$ Further $r=\min\{p,K\}$ with probability 1.  Then,
\begin{enumerate}
\item{}
$\{\bU, \bfOmega\}$ and $\bfV$ are statistically independent.
\item{}
$\bfV$ is a bi-unitarily invariant unitary matrix.
\end{enumerate}
\end{theorem}

\begin{proof}
1. We firstly show that $\{\bU, \bfOmega\}$ and $\bfV$ are statistically independent. 

Let $\bfU=[\bfU_0,\bfu_{r+1},\ldots,\bfu_p]=[\bfU_0 \,| \, \bfU_1]$ and let
$\bfV=[\bfV_0,\bfv_{r+1},\ldots,\bfv_K]=[\bfV_0 \,| \, \bfV_1].$
The full SVD $\bfJ=\bfU\bfPsi\bfV^H$ can be written in the form
$$
\bfJ=[\bfU_0 \,| \,  \bfU_1]\bfPsi \left[ \begin{array}{c} \bfV_0^H \\ \hline {\bfV}_1^H  \end{array} \right].
$$
Now consider two cases
\begin{itemize}
 \item{$K\leq p$}.
In this case, $r=K$ and 
\begin{equation}
\bfJ=[\bfU_0 \,| \,  \bfU_1]
\left[ \begin{array}{c} \bfOmega \\ \hline {\bf0}_{(p-K)\times K}  \end{array} \right]\bfV^H.
\label{eq:Kltp}
\end{equation}
\item{$K>p$}
In this case, $r=p$ and
\begin{equation}
\bfJ=\bfU
\left[ \begin{array}{c|c} \bfOmega & {\bf0}_{p\times (K-p)}  \end{array} \right]\left[ \begin{array}{c} \bfV_0^H \\ \hline {\bfV}_1^H  \end{array} \right].\label{eq:Kgtp}
\end{equation}
\end{itemize}

Write $\bfJ=\bfA+\eye \bfB.$ The probability density is given  by \cite[eqn.~78]{James64}
\begin{equation}\label{eq:pdfJ}
\pi^{-pK}|\bfS|^{-K}\exp^{-\tr\{\bfS^{-1}\bfJ\bfJ^H\}}{\textstyle{\prod}}_{i=1}^p{\textstyle{\prod}}_{j=1}^K \dif A_{ij} \dif B_{ij}.
\end{equation}
$\dif A_{ij}$ is the $i,j$-th element of $\dif\bfA$ and 
${\textstyle{\prod}}_{i=1}^p{\textstyle{\prod}}_{j=1}^K \dif A_{ij} \dif B_{ij}
$ is the volume element. Since we are interested in transforming $\bfJ$ it is convenient to use another notation for the volume element, viz  $(\dif \bfJ),$ 
so that (\ref{eq:pdfJ}) becomes
\begin{equation}\label{eq:pdftwo}
\pi^{-pK}|\bfS|^{-K}\exp^{-\tr\{\bfS^{-1}\bfJ\bfJ^H\}} (\dif \bfJ)
\end{equation}
which relates the volume element to the exterior product notation:
$$
(\dif \bfJ)\EqualDef (\dif\bfA)(\dif \bfB).
$$
where
$
(\dif \bfA)=\wedge_{j=1}^K \wedge_{i=1}^p \dif A_{ij};
$
see \cite[Chapter~2]{Muirhead82}.
Now we return to the case of $K\leq p$ and consider the `thin' SVD corresponding to (\ref{eq:Kltp}). It takes the form
\begin{equation}
\bfJ=\bfU_0 \bfOmega
\bfV^H.
\label{eq:Kltpthin}
\end{equation}

The transformation $\bfJ \rightarrow \bfU_0\bfOmega\bfV^H$  was studied in \cite{Ratnarajah05} who found the volume element  $(\dif\bfJ)$ to be proportional to
\begin{equation}\label{eq:volel}
 [\det\{\bfOmega\}]^{2p-2K+1}{\textstyle{\prod}}_{k<l}^K (\omega^2_k-\omega_l^2)^2 (\bfOmega)(\bfU_0\dif \bfU_0)(\bfV\dif \bfV^H).
\end{equation}
In (\ref{eq:pdftwo}), $\pi^{-pK}|\bfS|^{-K}\exp^{-\tr\{\bfS^{-1}\bfJ\bfJ^H\}}$
becomes
\begin{equation}\label{eq:pdfthree}
\pi^{-pK}|\bfS|^{-K}\exp^{-\tr\{\bfS^{-1}\bfU_0\bfOmega^2\bfU_0^H\}}.
\end{equation}
The product of (\ref{eq:pdfthree})  and the volume element (\ref{eq:volel}) shows that the probability density can be factored into functions of $\{ \bfU_0, \bfOmega\}$ and $\bfV.$ Now $\bfU=[\bfU_0 \,| \, \bfU_1]$ and in order for $\bfU$ to be unitary,
$\bfU_1$ depends totally on $\bfU_0.$
Hence $\bfV$ is independent of $\bfU$ and $\bfOmega.$

For the case $K>p$ consider the `thin' SVD corresponding to (\ref{eq:Kgtp}),
i.e., $\bfJ=\bfU \bfOmega
\bfV_0^H.$ Then the probability density can be factored into functions of $\{ \bfU, \bfOmega\}$ and $\bfV_0.$ Now $\bfV=[\bfV_0 \,| \, \bfV_1]$ and in order for $\bfV$ to be unitary,
$\bfV_1$ depends totally on $\bfV_0.$
Hence $\bfV$ is again independent of $\bfU$ and $\bfOmega.$
\end{proof}

2. We now show that the unitary matrix $\bfV$ is bi-unitarily invariant.

\begin{proof}
Note that
$
\bfJ^H\bfJ=\bfV\bfPsi^2\bfV^H=\bfV\bfLambda_K\bfV^H,
$ 
with
$$
\bfLambda_K=
\left[
\begin{array}{c|c}
\begin{matrix}
\lambda_1 & & \\
& \ddots & \\
&& \lambda_r
\end{matrix}
 & {\bf0}_{r \times (K-r)} \\ \hline
{\bf0}_{(K-r)\times r} & {\bf 0}_{(K-r)\times (K-r)}
\end{array}
\right].
$$
Since $\bfJ$ is  right-unitarily invariant (Lemma~\ref{lemma:one}) we know that $\bfJ$ and $\bfJ \bfTheta^H$ have the same distribution for $\bfTheta^H\in {\cal U}(K).$ Hence, with
${\displaystyle{\EqualDist}}$ denoting ``equal in distribution,''
$$\bfJ^H\bfJ\,{\displaystyle\EqualDist}\, (\bfJ\bfTheta^H)^H(\bfJ\bfTheta^H)= \bfTheta\bfJ^H\bfJ\bfTheta^H
=(\bfTheta\bfV)\bfLambda_K(\bfTheta\bfV)^H
$$
and so
$
\bfV\bfLambda_K\bfV^H \,{\displaystyle\EqualDist}\, (\bfTheta\bfV)\bfLambda_K(\bfTheta\bfV)^H.
$
The random components of $\bfLambda_K$  are functions of the random components of $\bfOmega,$ and  $\bfV$ is independent of $\bfU$ and $\bfOmega,$ 
so $\bfV$ and $\bfLambda_K$ are independent. Then,
$
\bfV \,{\displaystyle\EqualDist}\, \bfTheta\bfV.
$
Since the distribution of $\bfV$ is left-unitarily invariant and $\bfV\in {\cal U}(K),$ we know from Lemma~\ref{def:two} of
Appendix\ref{app:invariance} that it is also right-unitarily invariant, and hence is a bi-unitarily invariant unitary matrix.
This completes the proof.\end{proof}

\begin{lemma}\label{lemma:two}
With the $K\times K$ matrix $\bfV$ defined as in Theorem~\ref{theorem:main}, let $v_{jk}=(\bfV)_{jk}.$ Then for $1\leq j,k,l \leq K, j\not=l,$
\begin{eqnarray}
E\{ |v_{kj}|^4\}&=&2/[K(K+1)]\label{eq:Haione}\\
E\{ |v_{kj}|^2 \cdot|v_{kl}|^2\}&=& 1/[K(K+1)].\label{eq:Haitwo}
\end{eqnarray}
\end{lemma}
\begin{proof}
%Since $\bfV$ is bi-unitarily invariant, so is $\bfV^H.$ 
The bi-unitarily invariant nature of the unitary matrix $\bfV$ is sufficient \cite[p.~812]{HiaiPetz00}
for  the stated moment results of  \cite[Proposition~1.2]{HiaiPetz00} to hold, in particular (\ref{eq:Haione}) and (\ref{eq:Haitwo}).\end{proof}

\begin{lemma}\label{lemma:three}
We can write
 $$
E\left\{   \sum_{k=0}^{K-1}|| {\hat{\bfS}}_k-{\hat{\bfS}} ||^2_{\rm F}\,\, | \hbS\right\} =  \sum_{k=0}^{K-1} E\{ ||\bfJ_k||^4_2 | \hbS\}-K\tr\{
{\hat{\bfS}}^2 \}.
$$
\end{lemma}
\begin{proof}
Expanding the expectation on the left we get
\begin{eqnarray*}
 \sum_{k=0}^{K-1} E\{ \tr\{\bfJ_k\bfJ_k^H\bfJ_k\bfJ_k^H\}|\hbS\}-
 \sum_{k=0}^{K-1} E\{ \tr\{\hbS\bfJ_k\bfJ_k^H\}|\hbS\}\\
-\sum_{k=0}^{K-1} E\{ \tr\{\bfJ_k\bfJ_k^H\hbS\}|\hbS\}
+\sum_{k=0}^{K-1} E\{ \tr\{\hbS^2\}|\hbS\}
\end{eqnarray*}
Now,
$$
\tr\{\bfJ_k\bfJ_k^H\bfJ_k\bfJ_k^H\}=\tr\{\bfJ_k^H\bfJ_k\bfJ_k^H\bfJ_k\}=(\bfJ_k^H\bfJ_k)^2=||\bfJ_k||_2^4,
$$
so the first term is simply $ \sum_{k=0}^{K-1} E\{ ||\bfJ_k||^4_2 | \hbS\}.$
For the second term in the expansion we get 
$$
- E\{ \tr\{\hbS\sum_k\bfJ_k\bfJ_k^H\}|\hbS\}=- E\{ \tr\{ K\hbS^2\}|\hbS\}=-K\tr\{ \hbS^2\}.
$$
Terms three and four follow likewise to give the result.
\end{proof}

\begin{lemma}\label{lemma:four}
$$
 E\{ ||\bfJ_k||^4_2 | \hbS\}=\frac{K}{K+1}\left[\tr\{
{\hat{\bfS}}^2 \}+ \tr^2\{\hbS\}\right].
$$
\end{lemma}
\begin{proof}
We adopt the approach of \cite[Lemma~3]{Chen_etal10}, although details and the result are different.
Now 
\begin{equation}\label{eq:split}
K\hbS=\bfJ\bfJ^H=\bfU\bfPsi{\bfPsi}^H\bfU^H=\bfU\bfLambda_p\bfU^H,
\end{equation}
where, with $\lambda_i \in {\mathbb{R}},$
$$
\bfLambda_p=
\left[
\begin{array}{c|c}
\begin{matrix}
\lambda_1 & & \\
& \ddots & \\
&& \lambda_r
\end{matrix}
 & {\bf0}_{r \times (p-r)} \\ \hline
{\bf0}_{(p-r)\times r} & {\bf 0}_{(p-r)\times (p-r)}
\end{array}
\right].
$$
Let $\bfV^H=[\bfnu_0,\ldots,\bfnu_{K-1}]$ so that $\bfJ_k=
\bfU\bfPsi\bfnu_k$ and
$$
\bfJ_k^H\bfJ_k=\bfnu_k^H\bfPsi^H\bfPsi\bfnu_k=\bfnu_k^H\bfLambda_K \bfnu_k.
$$ 

Consequently,
\begin{eqnarray}
 \!\!\!\!\!\!\!E\{ ||\bfJ_k||^4_2 | \hbS\} &=&
E\{ (\bfnu_k^H\bfLambda_K \bfnu_k)^2|\hbS\}\nonumber\\
&=& 
E\{ E\{ (\bfnu_k^H\bfLambda_K \bfnu_k)^2|\hbS,\bfLambda_K\}|\hbS\}.
\label{eq:condind}
\end{eqnarray}
\begin{itemize}
\item$\hbS$ depends on $\bfU$ and $\bfLambda_p$ and the random components of $\bfLambda_p$  are functions of the random components of $\bfOmega.$
\item
The random components of $\bfLambda_K$  are functions of the random components of $\bfOmega.$
\item
 $\bfnu_k$ is a function of $\bfV$.
\end{itemize}
Now $\bfV$ is independent of $\bfU$ and $\bfOmega$ by Theorem~\ref{theorem:main}. 
Therefore, for the inner conditional expectation of (\ref{eq:condind}) we know that $E\{ (\bfnu_k^H\bfLambda_K \bfnu_k)^2|\hbS,\bfLambda_K\}$ is given by
\begin{eqnarray*}
&&\sum_{j=1}^r\lambda_j^2 E\{|\nu_{jk}|^4\} + \sum_{j\not=l}^r \lambda_j\lambda_l
E\{ |\nu_{jk}|^2|\nu_{lk}|^2\}\\
&&\qquad=
\sum_{j=1}^r\lambda_j^2 E\{|v_{kj}|^4\} + \sum_{j\not=l}^r \lambda_j\lambda_l
E\{ |v_{kj}|^2|v_{kl}|^2\}
\end{eqnarray*}
where $v_{kl}=(\bfV)_{kl}.$
Then using (\ref{eq:Haione}) and (\ref{eq:Haitwo}), we see that
\begin{eqnarray*}
E\{ (\bfnu_k^H\bfLambda_K \bfnu_k)^2|\hbS,\bfLambda_K\}
\!\!\!\!\!\!\!\!\!&&={{\frac{1}{K(K+1)}}}
\!\left[2\sum_{j=1}^r\lambda_j^2+\sum_{j\not=l}^r \lambda_j\lambda_l\right]\\
\!\!\!\!\!\!\!\!\!&&=\frac{1}{K(K+1)}
\left[\sum_{j=1}^r\lambda_j^2+\sum_{j,l}^r \lambda_j\lambda_l\right]\\
\!\!\!\!\!\!\!\!\!&&=\frac{1}{K(K+1)}\left[\tr\{\bfLambda_p^2\}+\tr^2\{\bfLambda_p\}\right]\\
\!\!\!\!\!\!\!\!\!&&=\frac{K}{K+1}\left[\tr\{\hbS^2\}+\tr^2\{\hbS\}\right],
\end{eqnarray*}
since from (\ref{eq:split}) we have that 
$$
\tr\{\bfLambda_p^2\}=K^2 \tr\{ \hbS^2\}\quad\mbox{and}\quad 
\tr^2\{\bfLambda_p\}=K^2 \tr^2\{ \hbS\}.
$$
Taking the outer expectation  conditional on $\hbS$ changes nothing,
which completes the proof.
\end{proof}


\section*{Acknowledgment}
The work of Deborah Schneider-Luftman was supported by EPSRC (UK). 
%The authors thank the Associate Editor and referees for several helpful %suggestions.

%% Can use something like this to put references on a page
%
%% by themselves when using endfloat and the captionsoff option.
%
%\ifCLASSOPTIONcaptionsoff
%
%  \newpage
%
%\fi
%
%
%
%
%
%
%
%% trigger a \newpage just before the given reference
%
%% number - used to balance the columns on the last page
%
%% adjust value as needed - may need to be readjusted if
%
%% the document is modified later
%
%%\IEEEtriggeratref{8}
%
%% The "triggered" command can be changed if desired:
%
%%\IEEEtriggercmd{\enlargethispage{-5in}}
%
%
%
%% references section
%
%
%
%% can use a bibliography generated by BibTeX as a .bbl file
%
%% BibTeX documentation can be easily obtained at:
%
%% http://www.ctan.org/tex-archive/biblio/bibtex/contrib/doc/
%
%% The IEEEtran BibTeX style support page is at:
%
%% http://www.michaelshell.org/tex/ieeetran/bibtex/
%
%%\bibliographystyle{IEEEtran}
%
%% argument is your BibTeX string definitions and bibliography database(s)
%
%%\bibliography{IEEEabrv,../bib/paper}
%
%%
%
%% <OR> manually copy in the resultant .bbl file
%
%% set second argument of \begin to the number of references
%
%% (used to reserve space for the reference number labels box)
%
%%\begin{thebibliography}{1}

\bibliographystyle{IEEEtran}

\end{document}